\documentclass[graybox]{svmult}
\usepackage{mathptmx}       
\usepackage{helvet}         
\usepackage{courier}        
\usepackage{makeidx}        
\usepackage{multicol}       
\usepackage{amssymb}

\newcommand{\re}{\mathrm{reg}}
\newcommand{\wiw}{\mathrm{w}}
\newcommand{\ma}{M\! A}
\newcommand{\sa}{S A}
\newcommand{\au}{\mathrm{Aut\,}}
\newcommand{\idn}{\mathfrak n}
\newcommand{\idm}{\mathfrak m}

\newcommand{\drk}{{\mathbb D}^r_k}
\newcommand{\id}{\mathrm{id}}
\newcommand{\idi}{\mathfrak i}
\newcommand{\idj}{\mathfrak j}
\newcommand{\orw}{\mathrm{ord}}

\newcommand{\mf}{\mathbf M\mathbf f}
\newcommand{\fm}{\mathbf F\mathbf M}
\newcommand{\cwa}{\mathbf W\mathbf A}
\newcommand{\er}{\mathbb R}
\newcommand{\aop}{\mathbb R[x_1,\dots,x_k]}

\makeindex             
\begin{document}
\title*{Fixed point subalgebras of Weil algebras: from geometric to algebraic questions}
\author{Miroslav Kure\v{s}}
\institute{Miroslav Kure\v{s} \at Institute of Mathematics, Department of Algebra and Discrete Mathematics, Brno University of Technology,
Technick\'a 2, 61669 Brno, Czech Republic \email{kures@fme.vutbr.cz}
\\
Published results were acquired using the subsidization of the
GA\v{C}R, grant No. 201/09/0981.}

\maketitle

\abstract{The paper is a survey of some results
about Weil algebras applicable in differential geometry, especially in some classification questions on bundles of generalized velocities
and contact elements. Mainly, a number of claims concerning the form of  subalgebras of fixed points of various Weil algebras is demonstrated.
}

\section{Introduction}
\label{sec:1}

Motivated by algebraic geometry, Andr\'e Weil suggested the treatment of infinitesimal objects as homomorphisms from algebras of smooth functions into some real finite-dimensional commutative algebra  with unit in 1950's. In fact, he follows
a certain idea of Sophus Lie: so-called $A$-near points (defined by Weil in \cite{WEI}) represent
the `parametrized infinitesimal submanifolds'. More precisely, let
$M$ be a smooth manifold and let $C^\infty(M,\er)$ be its ring of smooth functions into $\er$:
$A$-near points of~$M$ are defined as
$\er$-algebra homomorphism $C^\infty(M,\er)\to A$,
where $A$ is a certain local $\er$-algebra $A$ (precisely defined below) now called the Weil algebra.
This can be regarded as the first notable occurrence of local $\mathbb R$-algebras
 in differential geometry. 
New concepts, such as Weil algebras, Weil functors, Weil bundles
were introduced and they are widely studied, even to this day, because of their considerable generality.
In a modern categorical approach to differential geometry,
if we interpret geometric objects as bundle functors, then natural
transformations represent a number of geometric constructions. In this context,  finding a bijection between natural transformations of two Weil functors $T^A$, $T^B$
(generalizing well-known functors of higher order velocities and, of course,
the tangent functor as the first of them) and corresponding morphisms of Weil
algebras $A$ and $B$, has   fundamental importance. 
The theory of natural
bundles and operators, including methods for finding natural operators,
is very well presented in the monographical work Natural Operations in Differential Geometry \cite{KMS}
(Ivan Kol\'a\v{r}, Peter Michor and Jan Slov\'ak, 1993).
This paper has survey character: it provides an introduction to Weil algebras
and some selected problems which are geometrically motivated and were studied by the author and his collaborators from the algebraic point of view.

\section
{Starting points: product preserving functors}
\label{sec:2}

Let $F\colon \mf\to\fm$ be a bundle functor from the category
$\mf$ of manifolds (having smooth manifolds as objects and smooth maps as morphisms)
to the category $\fm$ of fibered manifolds (and fibered manifold morphisms). For example, such a functor is the tangent functor $T$.
For two manifolds $M_1$, $M_2$ we denote the standard projection onto
the $i$--th factor by $p_i\colon  M_1\times M_2\to M_i$,
where $i=1,2$.
$F$ is called {\bf product preserving} if the mapping
$$
(F(p_1),F(p_2))\colon F(M_1\times M_2)\to F(M_1)\times F(M_2)
$$
is a diffeomorphism for all manifolds $M_1$, $M_2$.
For a product preserving bundle functor we shall always identify
$F(M_1\times M_2)$ with $F(M_1)\times F(M_2)$ by the diffeomorphism
from the definition. The tangent functor $T$ is product preserving. Another example of a product preserving functor
is the functor $T^r_k$ of $k$-dimensional $r$-th order velocities withal $T^1_1=T$.
Further, we obtain a product preserving functor by arbitrary (finite)
iterations of product preserving functors.

If we denote by $\cwa$ the category of Weil algebras (the exact definition of Weil algebra is postponed to the next section) and Weil algebra
homomorphisms, then the problem of classification of all product preserving functors was solved in works of Kainz and Michor, Luciano and Eck
in the 1980's and reads as follows (see \cite{KMS}).\newline
{\it Product preserving bundle functors from the category
$\mf$ of manifolds into the category $\fm$ of fibered manifolds
are in bijection with objects of $\cwa$ and natural
transformations between two such functors
are in bijection with the morphisms of $\cwa$.}\newline
The correspondence is determined by the following construction
of the bundle functor $T^A$ from a given Weil algebra
$A$. Let $M$ be a smooth manifold and let
$A$ be a Weil algebra. Two smooth maps $g,h\colon\er^k\to M$ are said to
{\it determine the same $A$-velocity} $j^Ag=j^Ah$, if for every smooth
function $\phi\colon M\to\er$
$$
\pi_A(j^r_0(\phi\circ g))=\pi_A(j^r_0(\phi\circ h))
$$
is satisfied. (As usually, we denote here $r$-jets with the source in $0\in\er^k$ by $j^r_0$ and an epimorphism
from the algebra $\drk=J^r_0(\er^k,\er)$ to the algebra $A$ by $\pi_A$.)
The space $T^AM$ of all $A$-velocities on $M$ is fibered over $M$ and is called the {\it Weil bundle}. The functor $T^A$
from $\mf$ into $\fm$  is called
the {\it Weil functor}.

\section
{To the definition of the Weil algebra}
\label{sec:3}

The {\it Weil algebra} is a local commutative $\er$-algebra $A$
with identity, the nilradical (nilpotent ideal) $\idn_A$ of which has finite
dimension as a vector space and $A/\idn_A=\er$.
We call the {\it order} of $A$ the minimum $\orw (A)$ of the integers $r$
satisfying $\idn_A^{r+1}=0$ and the {\it width} $\wiw (A)$ of $A$
the dimension $\dim_\er(\idn_A/\idn_A^2)$.

One can assume $A$ is expressed as a finite dimensional factor
$\er$-algebra of the algebra $\aop$ of real polynomials in several
indeterminates. Thus, the main example is
$$
\drk=\aop/\idm^{r+1},
$$
$\idm=\langle x_1,\dots,x_k\rangle$ being the maximal ideal of $\aop$.
Evidently, $\orw(\drk)=r$ and $\wiw(\drk)=k$. Every other such algebra $A$ of order $r$ can be expressed in a form
$$
A=\aop/\idj=\aop/\idi+\idm^{r+1},
$$
where the ideal $\idi$ satisfies $\idm^{r+1}\subsetneqq\idi\subseteq\idm^2$
and is generated by a finite number of polynomials, i.e. $\idi=\langle P_1,\dots,P_l \rangle$.
The fact $\idi\subseteq\idm^2$ implies that the width of $A$ is $k$ as well.
It is evident, that such expressions of algebras in question are not unique
after all. Clearly, $A$ can be expressed also as
$$
A=\drk/\idi,
$$
where $\idi$ is an ideal in $\drk$.
This last definition will be prefered in the paper; we will also frequently move from $\drk/\idi$ to $\aop/\idj$ and back.

Let $\au A$ be the group of automorphisms of the algebra $A$.
By a {\it fixed point} of $A$ we mean every $a\in A$
satisfying $\phi(a)=a$ for all $\phi\in\au A$. Let
$$
\sa=\{a\in A; \phi(a)=a \;\mathrm{ for\; all }\;\phi\in\au A \}
$$
be the set of all fixed points of $A$.
It is clear, that
$\sa$ is a subalgebra of $A$ containing constants (of couse, every automorphism sends 1 into 1),
i.e. $\sa\supseteq\er$. If $\sa=\er$, we say that $\sa$ is {\it trivial}.

\section
{Weil contact elements}
\label{sec:4}

Now, let the Weil algebra $A$ have width $\wiw(A)=k<m=\dim M$ and order $\orw(A)=r$.
Every $A$-velocity $V$ determines an underlying
$\mathbb D^1_k$-velocity $\underline V$. We say $V$ is {\it regular}, if $\underline V$ is regular, i.e.
having maximal rank $k$ (in its local coordinates).
Let us denote $\re T^AM$ the open subbbundle of $T^AM$ of regular velocities on $M$.
The {\it contact element of type $A$} or briefly the {\it Weil contact element}
on $M$ determined by $X\in\re T^AM$ is 
the equivalence
class 
$$
\au A_M(X)=\{\phi(X);\phi\in\au A\}.
$$
We denote by $K^AM$ the set of all contact elements of type $A$
on $M$. Then
$$
K^AM=\re T^AM/\au A
$$
has a differentiable manifold structure and $\re T^AM\to K^AM$
is a principal fiber bundle with the structure group $\au A$.
Moreover, $K^AM$~is a generalization of the bundle of higher order contact elements 
$K^r_kM=\re T^r_kM/G^r_k$ introduced by Claude Ehresmann.
We remark that the local description of regular velocities and contact elements is covered by the paper
\cite{KAL}.

Let us write
$$
\epsilon_A\colon \au A\to \mathrm{GL}(\idn_A/\idn_A^2)
$$
for the canonical group morphism. 
If we write as usual $\wiw(A)=\dim_\er(\idn_A/\idn_A^2)$, then
$\mathrm{GL}(\idn_A/\idn_A^2)$ reads as $\mathrm{GL}(\wiw(A),\er)$.

Further, the element $\phi\in \au A$ is called {\it orientation preserving}, if
the determinant of $\epsilon_A(\phi)$ is positive.

The subgroup of all orientation preserving elements of $\au A$ will be denoted by
$(\au A)^+$.

If we factorize
$$
\re T^AM / (\au A)^+,
$$
we obtain the bundle $K^{A+}M$ of {\it Weil oriented contact elements}.

As to orientability, we remark that even the case $\au A=(\au A)^+$ can occur. So, it is suggestive
to study the orientability (with interesting references to classical geometric problems) just from the indicated point of view.

\section
{Subalgebra of fixed points}
\label{sec:5}

We use the fact that a Weil algebra $A$ can also be considered as a factor algebra of the algebra $\aop$ of polynomials, i.e. $A=\aop/\idj$
and then $\idj$=$\idi+\idm^{r+1}$,
where $\idm=\langle x_1,\dots,x_k \rangle$ is the maximal ideal in $\aop$. 
Let $\tau\in\er$, $\tau\ne 0$, and let $H_{\tau}\colon\aop\to\aop$ be a (linear diagonal) homomorphism acting by 
\begin{eqnarray*}
x_1 &\mapsto& \tau x_1\\
&\dots&\\
x_k &\mapsto& \tau x_k. 
\end{eqnarray*}
Then it is necessary to determine whether $H_{\tau}$ induces a homomorphism $\bar H_{\tau}\colon A\to A$ or not.

\begin{definition}
The Weil algebra $A=\drk/\idi$ is called {\it monomial}, if $\idi$ is monomial.
\end{definition}

\begin{proposition}
If $A$ is a monomial Weil algebra, then its subalgebra $\sa$ of fixed points is trivial.
\end{proposition}
\begin{proof}
\smartqed
It is clear that $\idj$ can be also generated by monomials. The homomorphism~$H_{\tau}$ sends every such monomial from $\idj$ again into $\idj$, i.e.
$H_{\tau}(\idj)\subseteq\idj$ and we have the induced homomorphism $\bar H_{\tau}\colon A\to A$. For $\tau\notin\{-1,0,1\}$,
$\bar H_{\tau}(a)\ne a$ for every element of $a\in\idn_A$. Thus, $\sa$ is trivial.
\qed
\end{proof}

\begin{definition}
The Weil algebra $A=\drk/\idi$ is called {\it homogeneous}, if $\idi$ is homogeneous.
\end{definition}

If we have a positive gradation $A=\bigoplus_{i\ge 0}A_i$ on a Weil algebra $A$ such that $\idn_A^n=\bigoplus_{i\ge n}A_i$ for each $n\ge 0$, we say
that $A$ is {\it gradable by the radical}, cf. \cite{SA1}.
We remark that $A$ is gradable by the radical if and only if $\idi$ is homogeneous.

\begin{proposition}
If $A$ is a homogeneous Weil algebra, then its subalgebra $\sa$ of fixed points is trivial.
\end{proposition}
\begin{proof}
\smartqed
The reason is completely identical to that in the previous proposition (see \cite{KM1} for the original proof):
the homomorphism~$H_{\tau}$ sends a homogeneous polynomial from $\idj$ again into $\idj$, i.e.
$H_{\tau}(\idj)\subseteq\idj$ and we have the induced homomorphism $\bar H_{\tau}\colon A\to A$. For $\tau\notin\{-1,0,1\}$,
$\bar H_{\tau}(a)\ne a$ for every element of $a\in\idn_A$ and $\sa$ is trivial.
\qed
\end{proof}

The idea of the proofs of the two propositions above lies in the fact that $H_{\tau}$ maps $\idj$ into $\idj$.
Thus, it is not difficult to derive the following slight generalization.
Let $\tau_1,\dots,\tau_k$ be non-zero real numbers and $H_{\tau_1,\dots,\tau_k}\colon\aop\to\aop$ be a (linear diagonal) homomorphism acting by 
\begin{eqnarray*}
x_1 &\mapsto& \tau_1 x_1\\
&\dots&\\
x_k &\mapsto& \tau_k x_k. 
\end{eqnarray*}

\begin{proposition}
If $A=\aop/\idj$ is a Weil algebra with $\wiw(A)=k$ and if there exist some $\tau_1,\dots,\tau_k\in\mathbb R-[-1,1]$
(or $\tau_1,\dots,\tau_k\in (-1,1)-\{0\}$) such that $H_{\tau_1,\dots,\tau_k}(\idj)\subseteq\idj$,
then the subalgebra $\sa$ of fixed points of $A$ is trivial.
\end{proposition}
\begin{proof}
\smartqed
The idea of the proof of this generalization is clear: if $H_{\tau_1,\dots,\tau_k}(\idj)\subseteq\idj$, then every non-constant monomial
from $\idj$ maps onto a monomial in $\idj$ (with the same multidegree), however, not onto the same monomial, because of the impossibility
to obtain 1 as a product of $\tau$'s. The induced homomorphism preserves this property.
\qed
\end{proof}

The assertions of the previous three propositions do not hold in the opposite direction --- not even the last one, which has the most general presumptions.
For example $A=\mathbb D^4_2/\langle x^2+y^3, x^3+y^4\rangle$ has trivial $\sa$, but there are no
$\tau_1,\tau_2\in\er-[-1,1]$ (or $\tau_1,\tau_2\in (-1,1)-\{0\}$)
such that $H_{\tau_1,\tau_2}(\langle x^2+y^3, x^3+y^4\rangle+\idm^5)\subseteq
\langle x^2+y^3, x^3+y^4\rangle+\idm^5$, see \cite{KM1}.

It is now the right time to show that there exist Weil algebras for which their subalgebras of fixed points are not trivial.
Examples of such algebras are
$\mathbb D^4_2/\langle x^2y+y^4,x^3+xy^2 \rangle$
or $\mathbb D^3_3 / \langle x^2+y^3, xy+z^3, y^2z+yz^2 \rangle$. This can be verified by a direct computation (although it is not evident at first sight: see Appendix!).
Moreover, the following "order theorem" holds.

\begin{proposition}
There is no algebra $A$ with $\wiw(A)=1$ and with nontrivial fixed point subalgebra.
There exist algebras $A$ with $\wiw(A)=2$ with a nontrivial fixed point subalgebra
if and only if $\orw(A)\ge 4$. For all 
$k>2$, there exist an algebra with $\wiw(A)=k$ and with a nontrivial fixed point subalgebra
if and only if $\orw(A)\ge 3$.
\end{proposition}
\begin{proof}
\smartqed
The proof is based on several technical lemmas and we do not write it here
for its length. We refer mainly to \cite{KUS} and also to \cite{KM2}.
\qed
\end{proof}

Let us follow through a slightly different but also fairly good approach. For a Weil algebra $A$, the canonical
algebra homomorphism $\kappa_A\colon A\to \er$ can be viewed as the endomorphism $\kappa_A\colon A\to A$.
The group $\au A$ of $\er$-algebra automorphisms of $A$ is a real smooth manifold with the usual Euclidean topology. Then the following
definition is correct.

\begin{definition}
A Weil algebra $A$ is said to be
{\it dwindlable} if there is an infinite sequence
$\{\phi_n\}_{n=1}^\infty$ of automorphisms $\phi_n\in\au A$ such that
$\phi_n\to\kappa_A$ for $n\to\infty$.
\end{definition}

\begin{proposition}
If $A$ is a dwindlable Weil algebra, then its subalgebra $SA$ of fixed
elements is trivial. Apart from that, there are non-dwindlable Weil algebras
with trivial $SA$.
\end{proposition}
\begin{proof}
\smartqed
If $A$ is dwindlable and $\sa$ is not trivial, then there exists an element $0\ne a\in\idn_A$ belonging to $\sa$.
As there is also an infinite sequence $\{\phi_n\}_{n=1}^\infty$,
$\phi_n\in\au A$, $\phi_n\to\kappa_A$ for $n\to\infty$, we deduce for $a$ that
$0\ne a=\phi_n(a)\to\kappa_A(a)=0$ which is a contradiction.
\qed
\end{proof}

On the other hand, $\mathbb D^5_2/\langle xy^2+x^5,x^2y+y^5\rangle$ represents an example of a non-dwindlable Weil algebra with trivial $\sa$.
Furthermore, let us remark that for a dwindlable Weil algebra $A$ the group $U_A$ of unipotent automorphisms
(i.e. such automorphisms $\phi$ for which $\id_A-\phi$ is a nilpotent
endomorphism of $A$) is a proper subgroup of the connected
identity component $G_A$ of $\au A$, see \cite{KAN}. The index of the subgroup $G_A$ also represents an important object of interest, cf.
\cite{KAS}.

Let us return to the geometric motivation. 
>From what we have stated, we have deduced in \cite{KM1} and \cite{KM2}
the following results: \newline
{\it There is a one-to-one correspondence between all natural operators lifting vector fields
from $m$-manifolds to the bundle functor $K^A$ of Weil contact
elements and the subalgebra of fixed elements $SA$ of $A$.}\newline
{\it There is a one-to-one correspondence between all natural affinors
on $K^A$ and the subalgebra of fixed elements $SA$ of $A$.}\newline
{\it All natural operators lifting 1-forms
from $m$-dimensional manifolds to the bundle functor $K^A$ of
Weil contact elements are classified for the case of dwindlable Weil algebras: they represent constant multiples of the vertical lifting.}

\subsection*{To the open problem}
We conclude that the main problem of an exact one-to-one characterization of Weil algebras having non-trivial fixed point subalgebras remains open.

Nevertheless, a number of partial (sub-)problems can be mentioned. For example, elements $a\in A$ annihilated
by any element of the nilradical $\idn_A$, i.e. having the property $au=0$ for all $u\in\idn_A$, 
constitute an ideal which is called the
{\it socle} of $A$ and denoted by $\mathrm{soc}(A)$.
Then elements of $A$ in the form
$r_1+r_2a$, $r_1,r_2\in\er$, $a\in\mathrm{soc}(A)$
form a subalgebra $\ma$ of $A$.
The problem of a relation between $\sa$ and $\ma$ is also open (with the conjecture: $\sa\subseteq \ma$).

\section*
{Appendix: The computation method and two examples}
\label{sec:6}

We present a computation method  for the description of automorphisms and  detecting whether the fixed point subalgebra is trivial or not.

\begin{example}
The first example is of theoretical importance, see Proposition~4.
Let 
$$
A=\mathbb D^4_2/\langle x^2y+y^4,x^3+xy^2 \rangle.
$$
The elements of $A$ have the form
$$
k_1+k_2 x+k_3 y+k_4 x^2+k_5 xy +k_6 y^2+k_7 x^3
+k_8 x^2y+k_9 y^3
$$
with the simultaneous vanishing of all monomials of the fifth or higher
order in common with $x^4$, $x^3y$, $x^2y^2$, $xy^3$, $x^2y+y^4$ and
$x^3+xy^2$. We shall describe automorphisms of $A$. 
Automorphisms preserve the unit; so, we determine them by saying what is mapped to $x$ and $y$, for clarity, denoted rather by $\bar x$ and $\bar y$.
Thus, the starting point is a form
\begin{eqnarray*}
\bar x &=
&Ax+By+Cx^2+Dxy+Ey^2+Fx^3+Gx^2y+Hy^3 \\
\bar y &=
&Ix+Jy+Kx^2+Lxy+My^2+Nx^3+Ox^2y+Py^3.
\end{eqnarray*}
The matrix
$
\left(\begin{array}{cc}      
A & B \\       
I & J \end{array}\right)
$
must be non-singular and we consider the conditions
$\bar x^4=0$, $\bar x^3\bar y=0$, $\bar x^2\bar y^2=0$, $\bar x\bar y^3=0$,
$\bar x^2\bar y+\bar y^4=0$ and $\bar x^3+\bar x\bar y^2=0$ now.
The condition $\bar x^4=0$ gives $B=0$. The conditions $\bar x^3\bar y=0$,
$\bar x^2\bar y^2=0$, $\bar x\bar y^3=0$ give no new nontrivial relation.
The condition $\bar x^2\bar y+\bar y^4=0$ gives $I=0$, $A^2=J^3$.
The condition $\bar x^3+\bar x\bar y^2=0$ gives $E=0$, $A^2=J^2$.
So, we obtain $J=1$ and $A=-1$ or $A=1$. Hence the automorphisms have the
following form
\begin{eqnarray*}
\bar x &=
&\epsilon x+Cx^2+Dxy+Fx^3+Gx^2y+Hy^3 \\
\bar y &=
&y+Kx^2+Lxy+My^2+Nx^3+Ox^2y+Py^3,
\end{eqnarray*}
where $\epsilon\in\{-1,1\}$. (We observe that the group $\au A$ of automorphisms has two connected components.)
Finally, we solve the equation
\begin{eqnarray*}
&k_1+k_2 \bar x+k_3 \bar y+k_4 \bar x^2+k_5 \bar x\bar y +k_6 \bar y^2+
k_7 \bar x^3 +k_8 \bar x^2\bar y+k_9 \bar y^3
&=\\
&k_1+k_2 x+k_3 y+k_4 x^2+k_5 xy +k_6 y^2+k_7 x^3
+k_8 x^2y+k_9 y^3 &
\end{eqnarray*}
for $k_i$, $i=1,\dots,9$, by using the described automorphisms.
By comparing coefficients at powers of $x$ and $y$,
we find that $k_2=k_3=k_4=k_5=k_6=k_7=k_9=0$ and $k_1,k_8$ are arbitrary
real coefficients. This means
$$
SA=\{k_1+k_8 x^2y;\; k_1,k_8\in\er\}
$$
and $SA\supsetneqq \er$. 
\end{example}

\begin{example}
The second example is new.
Let 
$$
A=\mathbb D^4_3/\langle x^2+y^3+z^3 ,x^3+y^3+z^4, xyz \rangle.
$$
We start by expressing of elements of $A$ in the form
\begin{eqnarray*}
&& k_1+k_2 x +k_3 y +k_4 z +k_5 x^2 +k_6 xy +k_7 y^2 +k_8 x z +k_9 yz +k_{10} z^2+\\
&& k_{11} x^2 y +k_{12} x y^2 +k_{13} x^2 z +k_{14} y^2 z +k_{15} x z^2 +k_{16} y z^2 +k_{17} z^3 +k_{18} y^2 z^2
\end{eqnarray*}
with the simultaneous vanishing of all monomials of the fifth or higher
in common with
$xyz$, $x^2y^2$, $x^2z^2$, $x^2+y^3+z^3$, $x^2-x^3+x^2 z+z^3$, $x^2 z+z^4$, $x^2 y+y z^3$, $x^2+x^2 z+z^3+x z^3$.
The algorithm given above yields after a "bit of calculation" a connected group of automorphisms (we leave its exact expression as an exercise to the reader)
and 
$$
SA=\{k_1+k_5 x^2+k_{12}x y^2+k_{13}x^2 z+k_{18}y^2 z^2;\; k_1,k_5,k_{12},k_{13},k_{18}\in\er\}.
$$
Hence, we find that the dimension of the subalgebra $SA$ of fixed points is remarkably high.
\end{example}

\section*
{Acknowledgement}
\label{sec:7}

The author thanks an unknown referee for comments that improved the paper.

\bibliographystyle{amsplain}
\makeatletter \renewcommand{\@biblabel}[1]{\hfill#1.}\makeatother

\end{document}